\documentclass[11pt]{amsart}
\usepackage{bm}
 \usepackage[margin=1.5in]{geometry}

\usepackage{multirow}
\usepackage{tikz}
\usepackage{caption}
\usepackage{textgreek}
\usetikzlibrary{
 decorations.pathmorphing,%
 decorations.markings,%
 decorations.pathreplacing,%
 decorations.text,%
 calc,%
 patterns,%
 shapes.geometric,%
 arrows,
 decorations.shapes,
 positioning,
 plotmarks,
 shadings,
 math,
 intersections
}

\definecolor{jeffColor}{RGB}{102, 0, 204}
\definecolor{yaizaColor}{RGB}{0, 153, 153}
\definecolor{certainty}{RGB}{64, 228, 198}
\definecolor{hope}{RGB}{228, 194, 64}

\definecolor{periodColor}{RGB}{255, 167, 105}
\definecolor{dark-green}{RGB}{135, 194, 130}
\tikzset{>=latex} 
\tikzset{font=\small}
\tikzset{mark size=1.5pt, mark options=thin}
\tikzset{pin distance=4pt,
 every pin edge/.style={<-, thin, shorten <= -2pt}}
\usepackage{array,booktabs,tabularx}
\usepackage{graphicx}
\usepackage{amssymb}
\usepackage{arydshln}
\usepackage{enumerate}
\usepackage{color}
\usepackage{esint}
\usepackage{mathtools}
\usepackage{dsfont}
\usepackage{ esint }
\usepackage{placeins}
\usepackage[show]{ed}
\usepackage[final]{showkeys}
\usepackage{mathrsfs}
\usepackage{changepage}   
\usepackage{comment}
\usepackage{enumitem}
\usepackage{soul}
\usepackage{needspace}
\usepackage[normalem]{ulem}

\usepackage{multicol}
\usepackage{tabto}
\usepackage{hyperref}

\definecolor{uipoppy}{RGB}{221,128,71}
\definecolor{uipaleblue2}{RGB}{179,196,215}
\definecolor{uiviolet}{RGB}{86,86,99}
\definecolor{uiblack}{RGB}{0, 0, 0}
\definecolor{azul}{RGB}{0,128,255}
\definecolor{verde}{RGB}{50,180,50}
\definecolor{pale-verde}{RGB}{155,207,145}
\definecolor{uipaleblue}{RGB}{108,199,220}
\definecolor{light-gray}{RGB}{198,198,198}

\newtheorem{lemma}{Lemma}
\newtheorem{theorem}{Theorem}
\numberwithin{theorem}{section}

\newtheorem{corollary}[lemma]{Corollary}
\newtheorem{proposition}[lemma]{Proposition}

\theoremstyle{definition}
\newtheorem{definition}[lemma]{Definition}
\newtheorem{remark}[lemma]{Remark}

\usepackage{dsfont}

\newcommand{\ep}{\varepsilon}

\newcommand{\mc}[1]{\mathcal{#1}}

\newcommand{\WFh}{\operatorname{WF_h}}

\newcommand{\RR}{\mathbb{R}}
\newcommand{\NN}{\mathbb{N}}
\newcommand{\abs}[1]{\ensuremath{\left| #1 \right|}}

\newcommand{\comp}{\operatorname{comp}}

\newcommand{\ang}[1]{{\langle{#1}\rangle}}

\def\XXint#1#2#3{{\setbox0=\hbox{$#1{#2#3}{\int}$} \vcenter{\hbox{$#2#3$}}\kern-.5\wd0}}

\DeclareMathOperator{\supp}{supp}

\newcommand{\e}{\varepsilon}

\numberwithin{equation}{section}
\numberwithin{lemma}{section}

\newcommand{\pa}{\partial}
\renewcommand{\Im}{\operatorname{Im}}

\newcommand{\CI}{\mathcal{C}^\infty}

\newcommand{\Pfac}{\Upsilon}

\newcommand{\Hyp}{\mc{H}}
\newcommand{\Gl}{\mc{G}_2}
\newcommand{\loc}{\operatorname{loc}}
\newcommand{\opt}{\operatorname{Op}^{\mathsf{T}}}
\newcommand{\op}{\operatorname{Op}}
\newcommand{\Psit}{\Psi_{\mathsf{T}}}
\newcommand{\psit}[1]{{}^{#1}\!\Psi_{\mathsf{T}}}
\newcommand{\sigmat}{\sigma_{\mathsf{T}}}

\newcommand{\Lap}{\Delta}
\newcommand{\secondff}{\mathrm{I\!I}}
\def\noqed{\renewcommand{\qedsymbol}{}}

\title[Propagation for singular potentials]{Propagation for Schr\"odinger operators with potentials singular along a hypersurface}

\author{Jeffrey Galkowski}
\address{Department of Mathematics, University College London, London, United Kingdom}
\email{j.galkowski@ucl.ac.uk}
\author{Jared Wunsch}
\address{Department of Mathematics, Northwestern University, Evanston, IL, USA}
\email{jwunsch@math.northwestern.edu}
\begin{document}

\begin{abstract}
In this article, we study propagation of defect measures for Schr\"odinger operators, $-h^2\Delta_g+V$, on a Riemannian manifold $(M,g)$ of dimension $n$ with $V$ having conormal singularities along a hypersurface $Y$ in the sense that derivatives along vector fields tangent to $Y$ preserve the regularity of $V$. We show that the standard propagation theorem holds for bicharacteristics travelling transversally to the surface $Y$ whenever the potential is absolutely continuous. Furthermore, even when bicharacteristics are tangent to $Y$ at exactly first order, as long as the potential has an absolutely continuous first derivative, standard propagation continues to hold.
\end{abstract}
\maketitle

\section{Introduction}

Let $(M,g)$ be a smooth Riemannian manifold and $Y\subset M$ be a smooth hypersurface. Let $V$ be real valued and smooth away from $Y$, with conormal singularities to $Y$ in the sense that derivatives along vector fields tangent to $Y$ preserve the regularity of $V$. In this article, we study propagation of singularities, as measured by semiclassical defect measures, for the Schr\"odinger operator
$$
P:=-h^2\Delta_g+V. 
$$

Let $p=|\xi|_g^2+V$ denote the semiclassical principal symbol of $P$ and $H_p$ the Hamiltonian vector field associated to $p$. Recall that a sequence of functions $u_h$ with $h \downarrow 0$ has a (not necessarily unique) defect measure $\mu$ if along some subsequence $h_j\downarrow 0$
$$
\langle \op(a)u_{h_j},u_{h_j}\rangle_{L^2(M)}\to \int_{T^*M}ad\mu
$$
for all $a\in C_c^\infty(T^*M)$. The standard propagation theorem for defect measures is that, if $V\in C^\infty(M;\mathbb{R})$,
$$
(-h^2\Delta_g+V)u=o(h)_{L^2},\qquad \|u\|_{L^2}=1,
$$
and $u$ has defect measure $\mu$, then $\mu$ is supported in the characteristic set of $p$ and is invariant under the Hamiltonian flow for $p$. On the other hand, if $V$ is not continuous along $Y$ e.g. has a jump singularity along $Y$, it can be shown that a positive proportion of the energy may reflect off of $Y$ (see \cite{Be:82}, \cite[Section 1.2]{GaWu:18}, \cite{DeLa:23}). In contrast, Theorem~\ref{t:hyp} below shows that, as long as $V$ is absolutely continuous at $Y$, there is no reflection along bicharacteristics transverse to $Y$. 

Before stating our results we introduce some classes of conormal potentials.
\begin{definition}
Let $\mathcal{B}$ denote a Banach space of functions on $M$.
We say that $V\in I\mc{B}(Y)$ if for all $k$ and  $X_1,\dots,X_k$ smooth vector fields tangent to $Y$, we have
$$
V\in \mc{B},\qquad X_1\dots X_kV\in \mc{B}.
$$
\end{definition}
In this paper, we will mainly focus on the cases $\mc{B}=W^{k,p}$ for $k=1,2$ and $p=1,\infty$.  Note that we phrase the hypotheses here in terms of conormal spaces, the real hypotheses of our main results can be expressed in a weaker formulation in terms of Fermi normal coordinates with respect to the hypersurface $Y$: Let $(x,y)$ be Fermi normal coordinates near $Y$ with $x$ the signed distance to $Y$. We usually need only that $V$ is smooth outside a Fermi neighborhood of $Y$ and, for $\mc{B}$ a Banach space of functions on $\mathbb{R}$,
$$
V \in \mc{B}(\mathbb{R}_x;C^\infty(\mathbb{R}^{n-1}_y)).
$$
In our initial formulations, we will give the coordinate-invariant versions of our theorems using conormal spaces; the weaker hypotheses under which we in fact prove these results are discussed near the end of this section.

We begin with a discussion of the existence of a bicharacteristic flow with low regularity assumptions.  Here we use heavily the structured nature of the potential; note that there are recent very strong results in the case of unstructured singular coefficients in \cite{BuDeLeRo:22}.

Let $f$ denote a defining function for the hypersurface $Y$, $\pi$ denote the projection map $\pi: T^*M \to M$, and let $\Hyp\subset T^*M$ denote the hyperbolic set
$$
\Hyp = \{p=0\} \cap (\{f\neq 0\} \cup \{H_p \pi^* f \neq 0\}),
$$
containing points off $Y$ as well as those points over $Y$ where the bicharacteristic flow is transverse to $Y$.
Let 
$$
\Gl = \{p=0\} \cap \{f=H_p \pi^* f=0,\ H_p^2\pi^* f \neq 0\}
$$
denote the points over $Y$ where the flow is ``glancing to exactly second order.''
\begin{theorem}\label{theorem:flow}
If $V \in IW^{1,1}$, then through every point in $\Hyp$ there exists a unique maximal integral curve of $H_p$ in $\Hyp$.  If $V \in IW^{2,1}$ then through every point in $\Hyp \cup \Gl$ there exists a unique maximal integral curve of $H_p$ in $\Hyp \cup \Gl$.
\end{theorem}
The integral curves in the hyperbolic region over $Y$ in general satisfy an ODE with merely $L^1$ coefficients; the solution is an absolutely continuous function of $t$ and the equation is satisfied weakly.  Note that this level of coefficient regularity ($L^1$) is below that required by the Peano existence theorem (continuity) hence the existence depends on the structure of the singularities.  Once we reach $IW^{2,1}$ regularity for $V$, by contrast, the coefficients of the Hamilton vector field are absolutely continuous hence we have existence by the Peano theorem, but not uniqueness; here again uniqueness is recovered from the particularities of the singularity structure.

We let $\varphi_t(\bullet)$ denote the Hamilton flow on either $\Hyp$ or on $\Hyp \cup \Gl$, according to the regularity of $V$.

We now turn to results on propagation of defect measure. We begin with 
a general low-regularity propagation result that holds for unstructured singular potentials (i.e., not yet employing the notion of conormality used above).  Note that this result is also obtainable from the (stronger) results of \cite{BuDeLeRo:22}; we include it here for the sake of completeness rather than novelty.
\begin{theorem}\label{theorem:C1}
Suppose that $V \in \mathcal{C}^1(M)$ is real valued and $u$ satisfies
$$
\|(-h^2\Delta_g+V)u\|_{L^2}=o(h),\qquad \|u\|_{L^2}\leq C,
$$
and has defect measure $\mu$. Then $\supp \mu\subset \{p=0\}$ and 
for all $a \in \CI_c(T^*M)$, $\mu(H_p(a))=0$.
\end{theorem}
Note that at this level of coefficient regularity (continuous), bicharacteristics exist but may fail to be unique in general, hence our conclusion concerns $H_p(a)$ but does not address the question of propagation along individual bicharacteristics or invariance under the (undefined) flow.  (In \cite{BuDeLeRo:22}, it is shown as a corollary that the support of $\mu$ must indeed be a union of the (non-unique) integral curves.)

In regions of $T^*M$ where the flow \emph{is} well defined, Theorem~\ref{theorem:C1} yields flow-invariance of the defect measure; in particular, coupled with Theorem~\ref{theorem:flow}, it yields the following result at hyperbolic points and second-order glancing points; here the strengthened regularity hypothesis on $V$ gives us existence of the flow:
\begin{corollary}\label{corollary:glancing}
Suppose that $V \in IW^{2,1}(Y)$ is real valued and $u$ satisfies
$$
\|(-h^2\Delta_g+V)u\|_{L^2}=o(h),\qquad \|u\|_{L^2}\leq C,
$$
and has defect measure $\mu$. If $B\subset\{p=0\}$ is Borel and 
$$
\inf_{0\leq t\leq T,\rho\in B}|H_p^2(\pi^*f)(\varphi_t(\rho))|+|H_p(\pi^*f)(\varphi_t(\rho))|+|\pi^*f(\varphi_t(\rho))|>0,
$$
then
$$
\mu(B)=\mu(\varphi_T(B)).
$$
\end{corollary}

As with the existence theory for integral curves, we can reduce the regularity assumptions for propagation of singularities to $W^{1,1}$ by assuming that the singularities of $V$ have the structure of a conormal distribution with respect to a hypersurface and that we restrict our attention to $\Hyp\subset T^*M$.
\begin{theorem}
\label{t:hyp}
Suppose $V\in IW^{1,1}(Y)$ be real valued and $u$ satisfies
$$
\|(-h^2\Delta_g+V)u\|_{L^2}=o(h),\qquad \|u\|_{L^2}\leq C,
$$
and has defect measure $\mu$. Then $\supp \mu\subset \{p=0\}$ and for all $B\subset \{p=0\}$ Borel and $T>0$ such that
$$
\inf_{0\leq t\leq T,\rho\in B}|H_p(\pi^*f)(\varphi_t(\rho))|+|\pi^*f(\varphi_t(\rho))|>0,
$$
we have 
$$
\mu(B)=\mu(\varphi_T(B)).
$$
\end{theorem}
Recall that the flow $\varphi_t:=\exp(tH_{p}):B\to T^*M$ is well defined for $t\in[0,T]$ by Theorem~\ref{theorem:flow}.



{As noted above, the results of this paper are established with weaker hypotheses on $V$ than the conormal ones stated above. The conormality hypotheses employed in the statements of the main theorems have the virtue of making invariant sense on a manifold (independent of metric) and of making contact with the hypotheses of the prior work \cite{GaWu:18}, where conormality plays an essential role.  Here, however, we can in fact get away with weaker hypotheses as follows (stated locally near $Y=\{x=0\}$ as the theorems are local in nature).  Theorem~\ref{theorem:flow} holds under the hypothesis that $V \in W^{1,1}(\RR_x; \CI(\RR_y^{d-1}))$ or\ $W^{2,1}(\RR_x; \CI(\RR_y^{d-1}))$ at $\Hyp$ and $\Hyp \cup \Gl$ respectively.  Corollary~\ref{corollary:glancing} then likewise requires only $W^{2,1}(\RR_x; \CI(\RR_y^{d-1}))$.  Finally, Theorem~\ref{t:hyp} requires only $V \in W^{1,1}(\RR_x; \CI(\RR_y^{d-1}))$.

The organization of this paper is as follows (note that it somewhat diverges from the order in which the results are stated above).  Theorem~\ref{theorem:flow} is proved in Section~\ref{section:bichar} below, with the hyperbolic and glancing versions of the theorem being respectively Lemma~\ref{l:hyp} and Lemma~\ref{l:gl}.  Theorem~\ref{theorem:C1}, which differs from the other main results presented here in having unstructured hypotheses on $V$, follows from the very general elliptic estimate Lemma~\ref{l:ellipticU} (to obtain $\supp \mu \subset\{p=0\}$) together with the propagation result Lemma~\ref{lemma:C1prop} from the last section of the paper.  Finally, Theorem~\ref{t:hyp}, dealing with hyperbolic propagation, follows from the elliptic estimate, Lemma~\ref{l:ellipticU}, coupled with the hyperbolic propagation estimate, Lemma~\ref{lemma:hypprop}.

Propagation of singularities for operators with conormal singularities has been studied both for wave equations~\cite{deHoUhVa:15} and in the semiclassical case~\cite{GaWu:18}. In~\cite{GaWu:18}, Gannot and the second author quantify the level (in terms of powers of $h$) at which singularities do not diffract off of $Y$ under stronger assumptions on the potential $V$ using sophisticated techniques from microlocal analysis. In contrast, the methods used in this article use only basic pseudodifferential calculus. We believe that the methods in this paper could also give the more refined estimates under less restrictive assumptions on $V$ than those in~\cite{GaWu:18}, but we do not pursue this here; instead aiming to give a relatively simple and accessible proof.

Propagation of singularities for rough metrics without structure assumptions has recently been investigated in~\cite{BuDeLeRo:22}, where the authors study  the question of null-controllability of the wave equation for $\mathcal{C}^1$ metrics.

\noindent\textsc{Acknowledgements.} 

\noindent The authors are grateful to two anonymous referees for helpful comments on the manuscript.  JG acknowledges support from EPSRC grants EP/V001760/1 and EP/V051636/1.  JW was partially supported
by Simons Foundation grant 631302, NSF grant DMS--2054424, and a
Simons Fellowship.

\noindent \textsc{Data Availability Statement.} This project generated no data.



\section{Conormal distributions}

We begin by clarifying our hypotheses on conormal regularity of $V$.  As above, let $Y \subset M$ be a smooth, embedded hypersurface and let $n=\dim M$.  Following H\"ormander \cite[Section 18.2]{Ho:85}, we recall the definition of conormal distributions $u \in I^m(M, Y)$ if for all $N \in \NN$ and all smooth vector fields $L_1,\dots L_N$ tangent to $Y$,
$$
L_1\dots L_N u \in {}^\infty H^{\text{loc}}_{(-m-n/4)}(M).
$$
We will not be concerned here with the specifics of the Besov space $H^{\text{loc}}_{(-m-n/4)}(M)$; rather, we note the equivalent definition (Theorem 18.2.8 of \cite{Ho:85}) that $u$ should be smooth away from $Y$ and that locally near $Y$, in coordinates $x,y$ in a collar neighborhood of $Y$ with $x$ a boundary defining function,
$$
u= \int_{\RR} e^{ix\xi} a(x,y,\xi) \, d\xi,
$$
where
$$
a \in S^{m+n/4-1/2}(\RR^{n}_{x,y}\times \RR_\xi)
$$
is a Kohn--Nirenberg symbol i.e.,
$$
a\in S^m(\RR^{n}_{x,y}\times \RR_\xi)\Longleftrightarrow  |\partial_{(x,y)}^\alpha\partial_\xi^\beta a(x,\xi)|\leq C_{\alpha \beta}\langle \xi\rangle^{m-|\beta|}.
$$
We now connect this well-known scale of spaces to the spaces of distributions arising in our hypotheses.
\begin{lemma}
Let $k \in \NN$.  Then for all $\ep>0$,
\begin{equation}\label{conormalinclusion}
I^{-n/4+1/2-k-\ep}(M, Y) \subset IW^{k,1}(Y) \subset I^{-n/4+1/2-k}(M,Y).
\end{equation}
\end{lemma}
\begin{proof}
All spaces of distributions above coincide with $\CI(M)$ locally away from $Y$, hence we work near $Y$ in local coordinates $(x,y)$.  Given $u(x,y)$ let $\check u(\xi,y)$ denote the partial Fourier transform in the $x$ variable. Then
$$
u \in IW^{k,1}(Y) \Longrightarrow  (\xi \pa_\xi)^j \pa_y^\alpha \xi^\ell \check u \in L^\infty(\RR^n_{x,y} \times \RR_\xi)\ \text{for all}\  \ell \leq k,\ j \in \NN, \alpha \in \NN^{n-1}.
$$
Thus,
$$
u \in IW^{k,1}(Y) \Longrightarrow  \check u \in S^{-k}(\RR^n \times \RR),
$$
yielding the second inclusion in \eqref{conormalinclusion}.

To get the first inclusion, note that for $I^{-n/4+1/2-k-\ep}(M, Y)$, $$\mathcal{F}^{-1}(x\pa_x)^j \pa_y^\alpha \pa_{x,y}^\beta u \in S^{-k+\abs{\beta}-\ep}$$ for all $j,\alpha,\beta$, hence if $\abs{\beta}\leq k$
$$
(x\pa_x)^j \pa_y^\alpha \pa_{x,y}^\beta u(x,y) = \int e^{ix\xi} b(x,y,\xi) \, d\xi
$$
for some $b \in S^{-\ep}$.  It now suffices to show the RHS is in $L^1$.  To this end, we make a splitting
$$
\int e^{ix\xi} b(x,y,\xi) \, d\xi=\int_{\abs{\xi}<\abs{x}^{-1-\ep}} e^{ix\xi} b(x,y,\xi) \, d\xi+\int_{\abs{\xi}\geq\abs{x}^{-1-\ep}} e^{ix\xi} b(x,y,\xi) \, d\xi\equiv W_<+W_>.
$$
Then $\abs{W_<}\leq C \abs{x}^{-1-\ep}$ since the integrand is bounded; this term is thus in $L^1$.  Integration by parts in $W_>$ using the vector field $x^{-1} D_\xi$ yields a boundary term $O(\abs{x}^{-1+\ep(1-\ep)})$ (also in $L^1$) plus an integral $$x^{-1}\int_{\abs{\xi}\geq\abs{x}^{-1-\ep}} e^{ix\xi} b'\, d\xi$$ with $b' \in S^{-1-\ep}$.  This latter term is bounded by $$C x^{-1}\int_{\abs{\xi}>\abs{x}^{-1-\ep}}^\infty \abs{\xi}^{-1-\ep} \, d\xi,$$ again yielding a term in $L^1$.  This establishes the first inclusion in \eqref{conormalinclusion}.

\end{proof}

Note that the spaces of conormal distributions $IW^{k,1}$, defined via testing by vector fields, have the virtue of being manifestly coordinate invariant; the spaces $W^{k,1}((-\delta, \delta); \CI(Y))$, defined locally in normal coordinates, are bigger, and generally suffice for our needs, but are not defined invariantly in the absence of a metric, nor globally away from $Y$.
\begin{lemma}
The inclusion $IW^{k,1}(M,Y)\subset W^{k,1}((-\delta,\delta);\CI(Y))$ holds and is continuous.
\end{lemma}
\begin{proof}
Let $u\in IW^{k,1}(Y)$. Since $u$ is smooth away from $Y$, we need only work locally near $Y$. Let $(x,y)$ be Fermi coordinates near $Y$ and $u\in IW^{k,1}(Y)$. Then, for each fixed $x$ and for each $j\leq k$, and $\beta \in \mathbb{N}^{n-1}$, the Sobolev emdedding in the $y$ variables yields
$$
\|\partial_x^j\partial_y^\beta  u(x,\cdot)\|_{L^\infty_y}\leq \|\partial_x^j\partial_y^\beta  u(x,\cdot)\|_{L^1_y} +\sum_{|\alpha|=n}\|\partial_y^{\alpha}\partial_x^j\partial_y^\beta u(x,\cdot)\|_{L^1_y},
$$
Integrating in $x$, we obtain 
$$
\| \partial_x^j\partial_y^\beta u \|_{L^1_xL_y^\infty}\leq \| \partial_x^j \partial_y^\beta 
 u\|_{L^1(M)}+\sum_{|\alpha|=n}\|\partial_y^{\alpha}\partial_x^j \partial_y^\beta u\|_{L^1(M)}<\infty,
$$
where the finiteness of the right-hand side follows from $u\in IW^{k,1}(M,Y)$. Since $\beta$ is arbitrary and $|j|\leq k$ is arbitrary, this implies the lemma.
\end{proof}

\section{On the bicharacteristic flow }\label{section:bichar}

In this section, we establish lemmas on the bicharacteristic flow that combine to prove Theorem~\ref{theorem:flow}.
\subsection{The bicharacteristic flow in the hyperbolic set}
We first consider the bicharacteristic flow along trajectories which pass through the hypersurface, $Y$, transversally. To do this, we employ normal coordinates with respect to the hypersurface $Y$, with $x$ denoting the signed distance to $Y$ and $y=(y_1,\dots, y_{d-1})$ tangential variables, so that the metric (which we recall is, by hypothesis, everywhere $\mathcal{C}^\infty$ and nondegenerate) takes the form
\begin{equation}\label{metric1}
g = dx^2 + h(x,y,dy)=dx^2 + h_{ij}(x,y) dy^i dy^j,
\end{equation}
with $h(x,y,dy)$ a smooth family in $x$ of metrics on $Y$.
The metric induces a dual metric on $T^* M$ given by
$$
\xi^2+h^{ij}(x,y) \eta_i \eta_j,
$$
with $h^{\bullet,\bullet}$ the inverse of $h$ in \eqref{metric1} and using  coordinates in which the canonical one-form is
$$
\xi dx + \eta \cdot dy.
$$
Then
$$
\sigma_h(-h^2\Lap_g) = \xi^2 + h^{ij}(x,y) \eta_i \eta_j
$$
and 
$$
p\equiv \sigma_h(P) = \xi^2-r
$$
with
\begin{equation}\label{r}
r=-V-h^{ij} \eta_i\eta_j.
\end{equation}
Hamilton's equations of motion now read
\begin{equation}\label{hamvf}
\begin{aligned}
\dot x&=2\xi\\
\dot \xi &= \frac{\pa r}{\pa x}\\
\dot y &= -\frac{\pa r}{\pa \eta}\\
\dot \eta &= \frac{\pa r}{\pa y}.
\end{aligned}
\end{equation}

\begin{lemma}
\label{l:hyp}
Suppose that $V\in W^{1,1}(\mathbb{R}_x;C^\infty(\mathbb{R}_y^{d-1}))$ and let $(x_0,\xi_0,y_0,\eta_0)\in \mathbb{R}^2\times \mathbb{R}^{2(d-1)}$ with $\xi_0\neq 0$. Then there are unique absolutely continuous functions $(x(t),\xi(t),y(t),\eta(t))$ solving~\eqref{hamvf} for almost every $t$ in a neighborhood of $t=0$ with initial data $(x_0,\xi_0,y_0,\eta_0)$.
\end{lemma}
\begin{proof}
We start by solving an auxiliary system of equations using the Carath\'eodory theory of ODEs. We will use $x$ as the independent variable since $\dot x=2\xi\neq 0$ in a neighborhood of $t=0$. 
Consider absolutely continuous functions $(x(t), \xi(t), y(t), \eta(t))$ solving the equation \eqref{hamvf} for almost every $t$ with initial data satisfying $\xi \neq 0$.  As long as $\xi \neq 0$, $\dot x =2\xi\neq 0$, with $x(t) \in \mathcal{C}^1$.  We may change the independent variable from $t$ to $x$ and equivalently solve
\begin{equation}\label{reparametrized}
\begin{aligned}
dt/dx &= (2\xi)^{-1}\\
d\xi/dx &= (2 \xi)^{-1} \frac{\pa r}{\pa x}\\
d y/dx &= -(2 \xi)^{-1} \frac{\pa r}{\pa \eta}\\
d\eta/dx &= (2 \xi)^{-1} \frac{\pa r}{\pa y}.
\end{aligned}
\end{equation}

Since $V\in W^{1,1}(\mathbb{R}_x;C^\infty(\mathbb{R}^{d-1}_y))$, and $g\in C^\infty$,  $\pa r/\pa y$ and $\pa r/\pa \eta$ are both in $W^{1,1}(\mathbb{R}_x;C^\infty(\mathbb{R}^{2(d-1)}_{(y,\eta)}))$ while $\pa r/\pa x \in L^{1}(\mathbb{R}_x;C^\infty(\mathbb{R}^{2(d-1)}_{(y,\eta)}))$. In particular, 
$$
\|\partial_{(y,\eta)}^\beta \pa r/\pa x\|_{L^1_x}\leq C_\beta.
$$
Hence, the mean value theorem yields the estimate
$$
\lvert \pa_x r(x,y_0,\eta_0)-\pa_x r(x,y_1,\eta_1)\rvert \leq C \sup_{y,\eta \in \Omega}  \lvert \nabla_{y,\eta} \pa_x r(x,y,\eta)\rvert \in L^1(\mathbb{R}_x)
$$
for all pairs $(y_0,\eta_0),(y_1,\eta_1)$ in a fixed neighborhood $\Omega$ of a given $(\overline{y},\overline{\eta}).$  A fortiori the same estimates (indeed, better ones) hold for $\pa_y r,$ $\pa_\eta r$.  Hence the hypotheses of the existence and uniqueness theorem of Carath\'eodory hold (see \cite[Theorem 5.3]{Ha:80}) and this theorem shows that there exists a unique solution to the equation with initial data in $(-\ep,\ep)_x \times \Omega$ and that the data-to-solution map is continuous.

Now, we simply define $x:(-\delta,\delta)_t\to \mathbb{R}$ by the inverse function of $t(x)$, which exists since $t$ is absolutely continuous with derivative bounded away from zero. Then the unique solution of~\eqref{hamvf} is given by $(x(t),\xi(x(t)),y(x(t)),\eta(x(t))).$
\end{proof}

\subsection{The bicharacteristic flow near glancing}

We now focus on trajectories which encounter $Y$ tangentially. In order to handle the flow in this setting, we will make some additional assumptions on the potential $V$ and the surface $Y$. Indeed, we assume that $V\in  W^{2,1}(\mathbb{R}_x;C^\infty(\mathbb{R}^{n-1}_y))$ and the surface $Y$ is (locally) such that for any defining function $f:M\to \mathbb{R}$ for $Y$,
\begin{equation}
\label{e:curved}
\{p=0,\ f=0,\ H_p\pi^*f=0,\ H_p^2\pi^*f= 0\}=\emptyset,
\end{equation}
where $\pi:T^*M\to M$ is the canonical projection. The assumption itself deserves a small comment since a priori $H_p^2$ is not well defined. However, since $\pi^*f$ depends only on the position variables in $M$, $H_Vf\equiv 0$ and we interpret $H_p^2\pi^*f$ as  
$$
H_p^2\pi^*f=(H_{|\xi|_g^2}+H_V)(H_{|\xi|_g^2}\pi^*f),
$$
which is well defined.

We include here an alternate characterization of the condition \eqref{e:curved} in terms of the Riemannian geometry of $Y$.
\begin{proposition}
The curvature condition \eqref{e:curved} is equivalent to the condition
\begin{equation}\label{curvature2}
\nabla_N V \notin \operatorname{conv}(2V k_1,\dots 2V k_{d-1})
\end{equation}
where $N$ denotes a choice of unit normal to $Y$, $k_j$ are the principal curvatures of $Y$, and $\operatorname{conv}$ denotes convex hull.
\end{proposition}
Note that a change of orientation of $Y$ changes the signs of both $N$ and of the principal curvatures, so that the condition \eqref{curvature2} is independent of orientation. 
\begin{proof}
The condition \eqref{e:curved} is equivalent to the condition that along the projection to $M$ of the Hamilton flow on the energy surface, we never have $x=0$, $\dot x=0$, $\ddot x=0$.  

As the Legendre transform of our Hamiltonian is $L(z,\dot z)=(1/4)\abs{\dot z}^2_g-V(z)$, the equations of motion in the base read (with $\nabla$ denoting the Levi-Civita connection)
$$
\nabla_{\dot\gamma}\dot \gamma = -2\nabla V.
$$
Thus if $\gamma(0)\in Y$ with $\dot \gamma(0)=v \in TY$ (i.e., $\dot x=0$), then (with $N=\nabla x$ denoting the oriented unit normal vector field)
$$\begin{aligned}
\ddot x&= \nabla_{v} \ang{\dot \gamma,N} \\ &= \ang{\nabla_v \dot \gamma, N}+\ang{v, \nabla_v N}\\ &= \ang{-2\nabla V, N}-\ang{\secondff(v,v),N},
\end{aligned}$$
where we have used the equation of motion (and the fact that $v=\dot \gamma$) in the final equality, and where $\secondff$ denotes the second fundamental form.

Thus, the condition \eqref{e:curved} is now equivalent to
\begin{equation}\label{tangency}
\ang{\secondff(v,v)+2 \nabla V,N} \neq 0
\end{equation}
for $v=\dot \gamma \in TY$.  Now owing to conservation of the Hamiltonian, we have $\abs{v}_g^2=-4V$, hence the equation \eqref{tangency} is equivalent to 
$$
\ang{2 V \secondff(\hat v, \hat v)-\nabla V,N} \neq 0
$$
where $\hat v$ is the unit vector in direction $v$.  The range of $\ang{\secondff(\bullet,\bullet), N}$ on the unit tangent space is the convex hull of the principal curvatures, hence the condition is that $\ang{\nabla V, N}$ not lie in the convex hull of $2V$ times these values.
\end{proof}
\begin{lemma}\label{l:gl}
Suppose that $V\in W^{2,1}(\mathbb{R}_x;C^\infty(\mathbb{R}^{n-1}_y))$ and $\rho_0=(x_0,\xi_0,y_0,\eta_0)$ satisfies 
$$
x_0=0,\ \xi_0=0,\qquad (H_p^2x)(\rho_0)=(H_p\xi)(\rho_0)\neq 0,
$$
then there is a neighborhood of $t=0$ such that the solution to~\eqref{hamvf} with initial condition $(x_0,\xi_0,y_0,\eta_0)$ exists and is unique.
\end{lemma}
\begin{proof} 
First, notice that
$$
(H_p\xi )(\rho_0)=\partial_x r(\rho_0)\neq 0.
$$
For $t>0$, define
\begin{gather*}
R_1(t):=\int_{-2|\partial_x r(\rho_0)|t^2}^{2|\partial_x r(\rho_0)|t^2}\sup_{y,\eta}(|\partial^2_x r(s,y,\eta)|)ds +|t|.
\end{gather*}
Observe that since $r\in W^{2,1}$, 
\begin{equation}
\label{e:to0}
\lim_{t\to 0}R_1(t)=0.
\end{equation}
Also note that $R_1(t)$ is increasing and strictly positive for $t>0$.

Since the right-hand side of~\eqref{hamvf} is continuous, absolutely continuous solutions to~\eqref{hamvf} exist; it remains to prove uniqueness.

Suppose that $\rho(t):=(\rho'(t),\xi(t))$ with $\rho'(t)=(x(t),y(t),\eta(t))$ is an absolutely continuous solution of~\eqref{hamvf} with $(x(0),\xi(0),y(0),\eta(0))=(0,0,y_0,\eta_0)=:\rho_0$. Then we claim that $\rho'(t)\in \mc{C}^{2}$, $\xi(t)\in \mc{C}^{1}$, and
\begin{equation}
\label{e:aprioriSol}
\begin{aligned}
x(t)&=\partial_xr(\rho_0)t^2+O(t^2R_1(t)),\\
\xi(t)&=\partial_xr(\rho_0)t+O(tR_1(t))\\
y(t)&=y_0-\partial_\eta r(\rho_0)t+\frac{t^2}{2}(\partial^2_{y\eta}r(\rho_0)\partial_\eta r(\rho_0)-\partial^2_{\eta}r(\rho_0)\partial_yr(\rho_0))+O(t^2R_1(t)),\\
\eta(t)&=\eta_0+\partial_yr(\rho_0)t+\frac{t^2}{2}(\partial^2_{y\eta}r(\rho_0)\partial_yr(\rho_0)-\partial^2_{y}r(\rho_0)\partial_\eta r(\rho_0))+O(t^2R_1(t)).
\end{aligned}
\end{equation}

First, observe that 
$$
\rho(t)=\rho(0)+\int_0^tF(\rho(s))ds,
$$
with $F\in W^{1,1}$. Therefore, since $\rho$ is continuous $\dot \rho\in \mc{C}^{0}$, and hence $\rho\in \mc{C}^{1}$. Next, observe that since $\xi\in \mc{C}^1$,
$$
\xi(t)=\partial_xr(\rho_0)t+o(t),
$$

Next, we consider $x(t)$. First observe that 
\begin{align*}
x(t)&=\int_0^t 2\xi(s)ds=\partial_xr(\rho_0)t^2+o(t^2),
\end{align*}
and, since $\xi\in \mc{C}^1$, $x(t)\in \mc{C}^2$.

We now prove the error estimates in~\eqref{e:aprioriSol}. Start by observing that
\begin{align*}
\xi(t)&=\int_0^t\partial_xr(\rho(w))dw\\
&=t\partial_xr(\rho_0)+\int_0^t\int_0^w \langle \nabla \partial_xr(x(s),y(x),\eta(s)),\dot\rho(s)\rangle \, ds\, dw\\
&=t\partial_xr(\rho_0)+\int_0^t\int_0^w \langle \nabla_{(y,\eta)} \partial_xr(x(s),y(s),\eta(s)),(\dot y(s),\dot \eta(s))\rangle \, ds\, dw\\
&\qquad +\int_0^t\int_0^w\partial^2_xr(x(s),y(s),\eta(s))\dot{x}(s)\, ds\, dw\\
&=t\partial_xr(\rho_0)+\int_0^t\int_0^w \langle \nabla_{(y,\eta)} \partial_xr(x(s),y(s),\eta(s)),(\dot y(s),\dot \eta(s))\rangle dsdw\\
&\qquad +\int_0^t\int_0^{x(w)}\partial^2_xr(z,y(s(z)),\eta(s(z)))\, dz\, dw.
\end{align*}
Here $s:[0,x(t)]\to [0,t]$ is the inverse of the map $x:[0,t]\to [0,x(t)]$. 
Therefore, for $|t|$ small enough, since $\rho\in \mc{C}^1$ and $|x(w)|= |\partial_x r(\rho_0)|w^2+o(w^2)$, 
\begin{equation}\label{R1}
\big|\xi(t)-t\partial_xr(\rho_0)\big|\leq CtR_1(t).
\end{equation}

Furthermore,
\begin{align*}
x(t)&=\int_0^t 2\int_0^s\partial_xr(\rho(w))dwds\\
&=t^2\partial_xr(\rho_0)+\int_0^t 2\int_0^s\int_0^w \langle \nabla \partial_xr(x(s'),y(s'),\eta(s')),\dot\rho(s')\rangle ds'dwds.
\end{align*}
Therefore, arguing as above and using again that $\rho\in \mc{C}^1$, 
$$
\big|x(t)-t^2\partial_xr(\rho_0)\big|\leq Ct^2R_1(t)
$$

For $y(t)$ we write
\begin{align*}
y(t)&=y_0-\int_0^t\partial_{\eta} r(\rho(s))ds\\
&=y_0-\partial_\eta r(\rho_0)t+\int_0^t\int_0^{w}\partial^2_{y\eta}r(\rho(s))\partial_\eta r(\rho(s))-\partial^2_{x\eta}r(\rho(s))2\xi(s)-\partial^2_{\eta}r(\rho(s))\partial_yr(\rho(s))\, ds\, dw.
\end{align*}
As $\partial_yr,\partial_\eta r,\partial^2_{y\eta}r,\partial^2_{\eta}r\in \mc{C}^{1}$, $\partial^2_{x\eta}r\in \mc{C}^{0}$, $\xi\in \mc{C}^{1}$, one then easily checks that $y\in \mc{C}^{2}$, and that the equation for $y(t)$ in~\eqref{e:aprioriSol} holds. The argument for $\eta(t)$ is identical.

We are now in a position prove the uniqueness of our solution. Suppose that $\rho_1(t)=(\rho_1'(t),\xi_1(t))$ and $\rho_2(t)=(\rho_2'(t),\xi_2(t))$ solve~\eqref{hamvf} with $\rho_1(0)=\rho_2(0)=(0,y_0,\eta_0,0)$, 
Then
\begin{equation}
\label{e:xi}
\begin{aligned}
|\dot\xi_1-\dot\xi_2|(t)&\leq |\partial_xr(\rho_1'(t))-\partial_xr(\rho_2'(t))|\\
&\leq \int_0^1   \partial^2_xr(\rho_1'(t)s+(1-s)\rho_2'(t))(x_1(t)-x_2(t))ds\\
&\qquad + \int_0^1 \langle \nabla_{(y,\eta)} \partial_xr(\rho_1'(t)s+(1-s)\rho_2'(t)),(y_1,\eta_1)(t)-(y_2,\eta_2)(t)\rangle ds\\
&\leq \Big(\int_0^1   |\partial^2_xr(\rho_1'(t)s+(1-s)\rho_2'(t))|ds + C\Big)|\rho_1'(t)-\rho_2'(t)|
\end{aligned}
\end{equation}
Next, observe that, since $|\partial (\partial_yr,\partial_\eta r)|\leq C$, 
\begin{align*}
|\dot \rho_1'(t)-\dot\rho_2'(t)|&\leq 2|\xi_1(t)-\xi_2(t)|+|(\partial_y r,\partial_\eta r)(\rho_1'(t))-(\partial_y r,\partial_\eta r)(\rho_2'(t))|\\
&\leq 2|\xi_1(t)-\xi_2(t)|+C|\rho_1'(t)-\rho_2'(t)|.
\end{align*}

In particular, for $0\leq t<1$, 
\begin{equation}
\label{e:estimateRho}
\begin{aligned}
&|\rho_1'(t)-\rho_2'(t)|e^{-Ct}\\
&\leq 2\int_0^te^{-Cs}|\xi_1(s)-\xi_2(s)|ds\\
&\leq  2\int_0^te^{-Cs}\int_0^s \int_0^1\Big(|\partial_x^2r(\rho_1'(w)s'+\rho_2'(w)(1-s'))|+C\Big)|\rho_1'(w)-\rho_2'(w)|\, ds'\, dw\, ds \\
&\leq  C\int_0^1\int_0^t\int_0^s \Big(|\partial_x^2r(\rho_1'(w)s'+\rho_2'(w)(1-s'))|+C\Big)|\rho_1'(w)-\rho_2'(w)|w^{-2}R_1^{-1}(w)w^2R_1(w)dw\, ds\, ds' \\
&\leq  C\||\rho_1'(\cdot)-\rho_2'(\cdot)|(\cdot)^{-2}R_1^{-1}(\cdot)\|_{L^\infty(0,t)}\int_0^1\int_0^t\int_0^s \Big(|\partial_x^2r(\rho_1'(w)s'+\rho_2'(w)(1-s'))|+C\Big)w^2R_1(w)dw\, ds\, ds'.
\end{aligned}
\end{equation}
for the final inequality, note that
\eqref{e:aprioriSol} implies that $\||\rho_1'(\cdot)-\rho_2'(\cdot)|(\cdot)^{-2}R_1^{-1}(\cdot)\|_{L^\infty(0,t)}\leq C<\infty$.
Now change variables, replacing $w$ by $z=s'x_1(w)+(1-s')x_2(w)$; let $w(z)$ denote the inverse map.  We further split $\rho'=(x,\rho'')$, i.e., $\rho''=(y,\eta)$.  Since 
$$
\inf_{s'\in[0,1]}|s'\dot{x}_1(w)+(1-s')\dot{x}_2(w)|\geq c|w|,
$$
we obtain
\begin{align*}
&\int_0^t\int_0^s \Big(|\partial_x^2r(\rho_1'(w)s'+\rho_2'(w)(1-s'))|w^2R_1(w)\, dw\, ds\\
&=\int_0^t\int_0^{x(s)} \Big(|\partial_x^2r(z,\rho_1''(w(z))s'+\rho_2''(w(z))(1-s'))|\frac{|w(z)|^2R_1(w(z))}{|s'\dot{x}_1(w(z))+(1-s')\dot{x}_2(w(z))|}\, dz\, ds\\
&\leq CtR_1(t)\int_0^t\int_0^{x(s)} |\partial_x^2r(z,\rho_1''(w(z))s'+\rho_2''(w(z))(1-s'))|\, dz\, ds\\
&\leq Ct^2[R_1(t)]^2.
\end{align*}

Using this in~\eqref{e:estimateRho} shows that for $0\leq t<1$,
\begin{equation}
\label{e:rhopEst}
|\rho_1'(t)-\rho_2'(t)|t^{-2}R_1^{-1}(t)\leq C_0\||\rho_1'(\cdot)-\rho_2'(\cdot)|(\cdot)^{-2}R^{-1}_1(\cdot)\|_{L^\infty(0,t)}R_1(t).
\end{equation}
 Hence~\eqref{e:rhopEst} and~\eqref{e:to0} implies that 
\begin{equation}
\label{e:left0}
\lim_{t\to 0^+}|\rho_1'(t)-\rho_2'(t)|t^{-2}R_1^{-1}(t)=0.
\end{equation}

Let 
$$
t_0:=\inf\{t>0\,:\, R_1(t)>C_0^{-1}\},
$$
and suppose there is $0<t_*<t_0$, with  $|\rho_1'(t_*)-\rho_2'(t_*)|>0$ ($t_0<\infty$ by~\eqref{e:to0}). Then, since $\rho_1',\rho_2'$ are continuous and \eqref{e:left0} holds, there is $0<t_{m}\leq t_*$ such that 
\begin{align*}
0<\||\rho_1'(\cdot)-\rho_2'(\cdot)|(\cdot)^{-2}R_1^{-1}(\cdot)\|_{L^\infty(0,t_*)}&=|\rho'_1(t_{m})-\rho'_2(t_{m})|t_{m}^{-2}R_1^{-1}(t_m)\\
&\leq C_0\||\rho_1'(\cdot)-\rho_2'(\cdot)|(\cdot)^{-2}R_1^{-1}(\cdot)\|_{L^\infty(0,t_{m})}R_1(t_m)\\
&\leq  C_0\||\rho_1'(\cdot)-\rho_2'(\cdot)|(\cdot)^{-2}R_1^{-1}(\cdot)\|_{L^\infty(0,t_{*})}R_1(t_*).
\end{align*}
Dividing by $\||\rho_1'(\cdot)-\rho_2'(\cdot)|(\cdot)^{-2}R_1^{-1}(\cdot)\|_{L^\infty(0,t_{*})}$, we obtain 
$$
1<C_0R_1(t_*),
$$
which is a contradiction since $R_1(t_*)<C_0^{-1}$. 

Thus, $\rho_1'(t)=\rho_2'(t)$ on $[0,t_0]$ and hence, from~\eqref{e:xi}, we have $\rho_1(t)=\rho_2(t)$ for $t\in [0,t_0]$. An identical argument applies for $t\in [-t_0,0]$.

\end{proof}

\section{Semiclassical preliminaries}

\subsection{Defect measures}

We recall here the notion of a defect measure. Let $h_n\to 0$ and $\{u(h_n)\}_{n=1}^\infty\subset L^2(M)$. For a Radon measure $\mu$ on $T^*M$, we say that \emph{$\mu$ is a defect measure for the family $\{u(h_n)\}$} if there exists a subsequence $h_{n_k}$ such that for all $a\in C_c^\infty(T^*M)$,
\begin{equation}\label{pure}
\lim_{j\to \infty }\langle \op (a)u(h_{n_k}),u(h_{n_k})\rangle _{L^2(M)}= \int a\, d\mu.
\end{equation}

We now recall the following fact about existence of defect measures (See~\cite[Theorem 5.2]{Zw:12}).
\begin{lemma}
\label{l:existence}
Suppose that $\sup_n\|u(h_n)\|_{L^2}<\infty$. Then there is a subsequence $n_k\to \infty$ and a positive Radon measure $\mu$ such that $u(h_{n_k})$ has defect measure $\mu$. 
\end{lemma}

\begin{remark}
By Lemma~\ref{l:existence}, defect measures exist for bounded families; they have no reason in general to be unique, however. Our results are formulated to apply to \emph{any} defect measure associated to a family of solutions to the Schr\"odinger equation.  By
Lemma~\ref{l:existence}, though, we may as well pass to a subsequence which is \emph{pure} in the sense that \eqref{pure} holds.  
We will do this freely from now on, henceforth restricting our attention to the pure subsequence.

Moreover, to ease notation below, we will often go further and drop the sequence notation entirely to simply say that $u$ has defect measure $\mu$, leaving the sequence implicit.
\end{remark}

\subsection{Tangential operators}
When analyzing our operators in the hyperbolic region, we will have cause to use a family of tangential pseudodifferential operators which we define here. For a concise treatment of semiclassical pseudodifferential operators, semiclassical wavefront set, and microsupport, we refer the reader to \cite[Appendix E]{DyZw:19}.
\begin{definition}
Given $\mc{W}$ a normed space of functions on $\mathbb{R}$, we write 
$$
\psit{\mc{W}}^m:=\{ \mc{W}(\mathbb{R}_x;\Psi^m(\mathbb{R}_y^{d-1}))\}.
$$
We let $\psit{\mc{W}}^{\comp}$ denote those families $A(x)$ where $\WFh (A(x))$ lies in a fixed compact subset of $T^* \mathbb{R}_y$ for all $x$. 

For symbols $a\in \mc{W}(\mathbb{R};S^m(\mathbb{R}^{2(d-1)})$, let
$$
\opt(a)u(x,y)=\frac{1}{(2\pi h)^{d-1}}\int e^{\frac{i}{h}\langle y-y',\eta\rangle}a(x,y,\eta)u(x,y)d\eta dy.
$$
We also define the symbol map $\sigmat:\psit{\mc{W}}^m\to \mc{W}(\mathbb{R};S^m(\mathbb{R}^{2(d-1)}))$ by 
$$
\sigmat(A)(x,y,\eta)=\sigma(A(x))(y,\eta).
$$
\end{definition}

We will require the following result about composing tangential operators with ordinary semiclassical pseudodifferential operators with compactly supported symbols.
\begin{lemma}\label{l:tangentialcomposition}
Let $a \in \CI_c((-\ep,\ep); \CI_c(T^*\RR^{d-1}))$ and $\chi\in \mathcal{C}_c^\infty(T^*M)$.  Then
$$
\opt(a) \op(\chi) \in \Psi_h^{\comp}
$$
and this operator has principal symbol $a\chi$.
\end{lemma}
\begin{proof}
This result follows from writing 
$$
\opt(a) = \op (a(x,y, \eta))
$$
where we view $a$ as a constant symbol in the $\xi$ variable (dual to $x$).  The composition is then an ordinary composition of pseudodifferential operators with symbols in $S(1)$ as in \cite[Section 4.4]{Zw:12}.
\end{proof}

\section{Technical estimates}
\subsection{Elliptic Estimates}
We start by giving elliptic estimates when the potential $V$ is only continuous.
\begin{lemma}
\label{l:elliptic}
Let $P=-h^2\Delta_g+V$ with $V\in \mathcal{C}^{0}$, and $p=|\xi|_g^2+V$.
Then for all $A\in \Psi_h^2(M)$ with $\WFh(A)\subset \{p\neq 0\}$, there is $C>0$ such that for all $u\in L^2(M)$ with $\limsup_{h\to 0}\|u\|_{L^2}<\infty$,
$$
\limsup_{h\to 0}\|Au\|_{L^2}\leq C\limsup_{h\to 0}\|Pu\|_{L^2}.
$$
\end{lemma}
\begin{proof}
Since $V\in \mathcal{C}^{0}$, there is $V_\e\in C^\infty$ such that 
$$
 \lim_{\e\to 0^+}\|V_\e - V\|_{\mathcal{C}^{0}}=0.
$$
Indeed, one can construct $V_\e$ locally as follows. Let $\{\chi_i\}_{i=1}^N\subset C_c^\infty(M)$ be a partition of unity on $M$ with $\supp \chi_i\subset U_i$ and $(\psi_i:U_i\to \mathbb{R}^d,U_i)$ a coordinate system on $M$. Then, let $\phi \in C_c^\infty(\mathbb{R}^d)$ with $\int \phi=1$, define $\phi_\e(x):=\e^{-d}\phi(\e^{-1}x)$, and put 
$$
V_\e:= \sum_i\big[(\chi_i V)\circ \psi_i^{-1}*\phi_\e\big]\circ \psi_i
$$

Now, let $P_\e=-h^2\Delta_g+V_\e$. Then, for $A$ with $\WFh(A)\subset \{p\neq 0\}$, we have for $\e$ small enough $\WFh(A)\subset \{|p_\e|=||\xi|_g^2+V_\e|>c>0\}$ and hence, by the standard elliptic estimate~\cite[Theorem E33]{DyZw:19},
\begin{align*}
\limsup_{h\to 0}\|Au\|_{L^2}&\leq C\limsup_{h\to 0}\|P_\e u\|_{L^2}\\
&\leq C\limsup_{h\to 0}\|Pu\|_{L^2}+C\limsup_{h\to 0}\|(V-V_\e)u\|_{L^2}\\
&\leq C\limsup_{h\to 0}\|Pu\|_{L^2}+C\|V-V_\e\|_{\mathcal{C}^0}\limsup_{h\to 0}\|u\|_{L^2}.
\end{align*}
Since the left-hand side is independent of $\e>0$, this implies the lemma after sending $\e\to 0$.

 \end{proof}

 Before stating our next lemma, we recall that a pure sequence is one along which $\ang{\op a u_h,u_h}$ converges to $\mu(a)$ for a unique defect measure.


\begin{lemma}
\label{l:ellipticU}
Suppose that $V\in \mathcal{C}^0$, $\|u\|_{L^2}\leq C$ and $Pu=(-h^2\Delta_g+V)u=o(1)_{L^2}$. Then there is $\chi \in C_c^\infty(T^*M)$ such that 
\begin{equation}
\label{e:prelimElliptic}
\|\op(1-\chi)u\|_{H_h^2}=o(1). 
\end{equation}
Furthermore, if $u_h$ is a pure sequence with defect measure $\mu$, then $\supp \mu\subset \{p=0\}$ and 
$$
\mu(\{p=0\})=\lim_{h\to 0}\|u\|_{L^2}^2.
$$
\end{lemma}
\begin{proof}
Let $\chi \equiv 1$ near $\{p=0\}$. Then, we apply Lemma~\ref{l:elliptic} with $A=\op(\langle\xi\rangle^2(1-\chi))$ to obtain~\eqref{e:prelimElliptic}. 

To see that $\supp \mu\subset \{p=0\}$, let $a\in C_c^\infty(T^*M)$ with $\supp a\subset \{p\neq 0\}$. Then, by Lemma~\ref{l:elliptic}
$$
\|\op(a)u\|_{L^2}\leq C\|Pu\|_{L^2}+o(1)=o(1).
$$
In particular, 
$$
|\mu(a)|=|\lim_{h\to 0} \langle \op(a)u,u\rangle|\leq C\limsup_{h\to 0}\|\op(a)u\|_{L^2}=0.
$$

Finally, to see that $\mu(\{p=0\})=\lim_{h\to 0}\|u\|_{L^2}$, observe that, by~\eqref{e:prelimElliptic}
$$
\limsup_{h\to 0}\|u\|_{L^2}^2=\limsup_{h\to 0}\langle u,u\rangle=\limsup_{h\to 0}\langle \op(\chi)u,u\rangle=\mu(\chi).
$$
Similarly, 
$$
\liminf_{h\to 0}\|u\|_{L^2}^2=\mu(\chi)
$$
and hence, $\lim_{h\to 0}\|u\|_{L^2}^2=\mu(\chi)$ which implies the final claim.
\end{proof}

\subsection{$L^\infty L^2$ estimates}
In this section and the following, we discuss the factorization of semiclassical operators in the hyperbolic set near the interface $Y$ and its consequences; for related computations we also refer the reader to \cite[Section 2]{Ch:11}. 

To begin, we explore the consequences for energy estimates of having any operator in factorized form. In this section, we consider two (potentially different) factorized operators $\Pfac_{\pm}$:
\begin{equation}\label{factorization}
\begin{gathered}
\Pfac_+= (hD_{x}-\Lambda_0)(hD_x+\Lambda_0)+ hE_{+},\\
\Pfac_-=(hD_{x}+\Lambda_0)(hD_x-\Lambda_0)+ h E_{-},
\end{gathered}
\end{equation}
where $E_-,E_+\in \psit{L^1}^1$, and $\Lambda_0\in \psit{L^\infty}^1$ with real valued principal symbol.  In practice, the operators $\Pfac_{\pm}$ will be nearly equal.

\begin{lemma}
\label{l:LinfL2}
Suppose that $\Lambda_0(x)$ is elliptic on $\WFh(u(x))$ for all $x\in (-2\e,2\e)$. Then there is $C>0$ such that 
$$
\|hD_xu\|_{L^\infty((-\ep,\ep)_x)L^2_y}+\|u\|_{L^\infty((-\ep,\ep)_x)L^2_y}\leq C(\|u\|_{H_{h}^1}+h^{-1}(\|\Pfac_+ u\|_{L^2}+\|\Pfac_- u\|_{L^2})).
$$
\end{lemma}
\begin{proof}
Suppose that $\Pfac_\pm u=f_\pm$ and put \begin{equation}\label{vpm} v_{\pm}=(hD_x\pm\Lambda_0)u\end{equation} so that 
$$
(hD_x\mp\Lambda_0)v_{\pm}= f_{\pm}-h E_{\pm} u.
$$
For $\chi \in C_c^\infty((-2\e,2\e))$, equal to $1$ on $(-\ep, \ep)$,
$$
(hD_x\mp\Lambda_0)\chi(x)v_\pm=\chi f_{\pm}+[hD_x,\chi] v_{\pm}-h\chi E_{\pm}u.
$$

Therefore, 
\begin{align*}
\|\chi(x)v_{\pm}(x)\|^2_{L_y^2}&= -ih^{-1}\int_x^\infty hD_s\|\chi(s)v_{\pm}(s)\|^2_{L_y^2}ds\\
&=h^{-1}\int_x^\infty 2\Im \langle \pm \Lambda_0 \chi(s)v_{\pm}(s),\chi(s)v_{\pm}(s)\rangle_{L^2_y}+2\Im\langle \chi f_{\pm}(s),\chi(s)v_{\pm}(s)\rangle_{L^2_y}\\
&\qquad +2\Im \langle i^{-1}h\chi'(s)v_{\pm}(s),\chi(s)v_{\pm}(s)\rangle-2\Im\langle h\chi(s)E_{\pm}u(s),\chi(s)v_{\pm}(s)\rangle_{L^2_y}ds\\
&\leq C\|v_{\pm}\|^2_{L^2}+Ch^{-2}\|f_{\pm}\|^2_{L^2} + C\int_x^\infty g(s)(\|\chi(s)v_{\pm}(s)\|_{L_y^2}^2+\|\chi(s)u(s)\|_{H_{h,y}^1}^2)ds,
\end{align*}
where $g(s)\in L^1_{\loc}$, and where we have used the fact that $\Lambda_0$ has real principal symbol in our estimate of the corresponding term above. Considering $v_{+}-v_-$ and using ellipticity of $\Lambda_0$ on $\WFh(u(x))$, we have
\begin{align*}
\|\chi(x)u(x)\|^2_{H_{h,y}^1}&\leq C(\|\chi(x)v_{+}(x)\|^2_{L_y^2}+\|\chi(x)v_{-}(x)\|^2_{L_y^2} +O(h^\infty)\|u(x)\|_{L^2_y})\\
&\leq C(\|v_{+}\|^2_{L^2}+\|v_-\|_{L^2}^2+O(h^\infty)\|u\|_{H_h^1}+Ch^{-2}(\|f_+\|^2_{L^2} +\|f_-\|_{L^2}^2\\
&\qquad+ C\int_x^\infty g(s)(\|\chi(s)v_{+}(s)\|_{L_y^2}^2+\|\chi(s)v_{-}(s)\|_{L_y^2}^2+\|\chi(s)u(s)\|_{H_{h,y}^1}^2)ds.
\end{align*}
So, since $\|v_{\pm}\|_{L^2}\leq C\|u\|_{H_h^1}$, all together, we have
\begin{align*}
&\|\chi(x)u(x)\|^2_{H_{h,y}^1}+\|\chi(x)v_+(x)\|^2_{L_y^2}+\|\chi(x)v_-(x)\|^2_{L_y^2}\\
&\leq C(\|u\|_{H_h^1}^2+Ch^{-2}(\|f_+\|^2_{L^2}+\|f_-\|_{L^2}^2) \\
&\qquad+ C\int_x^\infty g(s)(\|\chi(s)v_{+}(s)\|_{L_y^2}^2+\|\chi(s)v_{-}(s)\|_{L_y^2}^2+\|\chi(s)u(s)\|_{H_{h,y}^1}^2)ds.
\end{align*}
Hence, by Gr\"onwall's inequality, for all $x$,
\begin{equation*}
\|\chi(x)u(x)\|^2_{H_{h,y}^1}+\|\chi(x)v_+(x)\|^2_{L_y^2}+\|\chi(x)v_-(x)\|^2_{L_y^2}
\leq C(\|u\|_{H_h^1}^2+h^{-2}(\|f_+\|^2_{L^2}+\|f_-\|^2_{L^2}))e^{C\|g\|_{L^1}},
\end{equation*}
and we obtain the desired pointwise estimate on $u$ for $x \in (-\ep,\ep)$, where $\chi=1$.

The estimate on $hD_xu$ now follows since 
$$
\|hD_x u(x)\|_{L^2_y}\leq C(\|u(x)\|_{H_{h,y}^1}+\|v_+(x)\|_{L^2_y}).\qed
$$\noqed
\end{proof}

\section{Estimates for the Schr\"odinger equation}


We now consider defect measures for solutions to Schr\"odinger equations with low regularity, conormal potentials. In particular, we assume that $V\in W^{1,1}((-2\ep,2\ep); \CI(Y))$ and use Fermi normal coordinates relative to $\{x=0\}$ so that in the notation of \eqref{r},
$$
P= (hD_x)^2-r(x,y,hD_{y})+h\big(a(x,y)hD_x +r'(x,y,hD_y)\big),
$$
with $r\in IW^{1,1}(\{x=0\};S^2(T^*\mathbb{R}^{d-1}))$, $r'\in \mc{C}^\infty(\mathbb{R}_x;S^1(T^*\mathbb{R}^{d-1}))$, and $a \in \CI$.
Now conjugate by $e^{\frac{i}{2}\int_0^xa(s,y)ds}$ to obtain 
$$
\Pfac:=e^{\frac{i}{2}\int_0^xa(s,y)ds}Pe^{-\frac{i}{2}\int_0^xa(s,y)ds}=(hD_x)^2-r(x,y,hD_{y}) +h\tilde{a}(x,y,hD_y),
$$
with $\tilde{a}\in C^\infty(\mathbb{R}_x;S^1(T^*\mathbb{R}^{d-1}))$.

\begin{lemma}
\label{l:factorization}
Let $V \in W^{1,1}((-2\ep,2\ep); \CI(Y))$.  Suppose that $\chi \in C_c^\infty((-2\e,2\e); C_c^\infty(T^*\mathbb{R}^{d-1}))$ with $\supp(\chi)\subset \{r>0\}$. Then there is $\Lambda \in \psit{W^{1,1}}^1$ satisfying
\begin{align*}
\Pfac\opt(\chi)&= (hD_{x}-\Lambda)(hD_x+\Lambda)\opt(\chi)+hE_+\opt(\chi)+O(h^\infty)_{\psit{W^{1,1}}^{-\infty}}\\
&= (hD_{x}+\Lambda)(hD_x-\Lambda)\opt(\chi)+hE_-\opt(\chi)+O(h^\infty)_{\psit{W^{1,1}}^{-\infty}},
\end{align*} 
with $E_\pm\in \psit{L^1}$ and the symbol $\Lambda$ satisfying
$$
\sigma(\Lambda)\chi= \sqrt{r(x,y,\eta)}\chi(x,y,\eta).
$$
\end{lemma}
\begin{proof}
Let $\tilde{\chi}\in C_c^\infty((-2\e,2\e); C_c^\infty(T^*\mathbb{R}^{d-1}))$ with $\tilde{\chi}\equiv 1$ on $\supp \chi$ and $\supp \tilde{\chi}\subset \{r>0\}$. Put $\Lambda=\opt(\sqrt{r}\tilde{\chi})\in \psit{W^{1,1}}^{\comp}$. 
Then,
$$
(hD_x-\Lambda)(hD_{x}+\Lambda)=((hD_x)^2-\Lambda^2+[hD_x,\Lambda])=(hD_x)^2-\opt(r\tilde{\chi}^2)+O(h)_{\psit{W^{1,1}}^{\comp}} +[hD_x,\Lambda].
$$
Now, since $\Lambda \in \psit{W^{1,1}}^1$, 
$$[hD_x,\Lambda]\in h\psit{L^1}^1.$$

Next, observe that, since $\tilde{\chi}\equiv 1$ on $\supp \chi$, 
$$
\Pfac\opt(\chi)=[(hD_x)^2-\opt(r\tilde{\chi}^2)]\opt(\chi)  +O(h^\infty)_{\psit{W^{1,1}}^{-\infty}}
$$
In particular, 
$$
\Pfac\opt(\chi)=(hD_x-\Lambda)(hD_{x}+\Lambda)\opt(\chi) +hE_+\opt(\chi)+O(h^\infty)_{\psit{W^{1,1}}^{-\infty}},
$$
with $E_+\in \psit{L^1}^1$.

An identical argument shows that 
$$
\Pfac\opt(\chi)=(hD_x+\Lambda)(hD_{x}-\Lambda)\opt(\chi) +hE_-\opt(\chi)+O(h^\infty)_{\psit{W^{1,1}}^{-\infty}},
$$
with $E_-$ as claimed.
\end{proof}

\begin{lemma}
\label{l:hyperbolicLinfty}
Suppose that $X\in \psit{C^\infty}^0$ with $\WFh(X)\subset \{r>0\}$. Then
$$
\|hD_x Xu\|_{L^\infty((-\ep,\ep)_x) L^2_y}+\|Xu\|_{L^\infty ((-\ep,\ep)_x) H^1_{h,y}}\leq C(\|u\|_{L^2}+h^{-1}\|P u\|_{L^2}).
$$
\end{lemma}
\begin{proof}
First, observe that $\sigma(\Pfac)=\xi^2-r(x,y,\eta)$ is elliptic for $|(\xi,\eta)|$ large enough. Therefore, by Lemma~\ref{l:elliptic} there is $\chi \in C_c^\infty(T^*\mathbb{R}^d)$ such that 
$$
\|\op(1-\chi)v\|_{H_h^2}\leq C(\|v\|_{L^2}+\|\Pfac u\|_{L^2}).
$$
In particular, 
\begin{equation}
\label{e:basicEllipticP}
\|v\|_{H_h^2}\leq C(\|v\|_{L^2}+\|\Pfac v\|_{L^2}+\|\op(\chi)v\|_{H_h^2})\leq C(\|v\|_{L^2}+\|\Pfac v\|_{L^2}).
\end{equation}

Now, let $\tilde{X}= e^{\frac{i}{2}\int_0^xa(s,y)ds}Xe^{-\frac{i}{2}\int_0^xa(s,y)ds}\in \psit{C^\infty}^0$
and observe that $\WFh(\tilde{X})=\WFh(X)\subset \{r>0\}\subset \operatorname{ell}(\Lambda)$, hence, there is  $\chi \in C_c^\infty((-2\e,2\e); C_c^\infty(T^*\mathbb{R}^{d-1}))$ with $\chi\equiv 1$ on $\WFh(\tilde{X})$ and $\supp \chi \subset \{r>0\}$. In particular,
$$
\tilde{X} v=\opt(\chi) \tilde{X} v + O(h^\infty)_{\psit{C^\infty}^{-\infty}} v.
$$
Consequently, Lemma~\ref{l:factorization} now implies that $\tilde{X} v$ satisfies equations
\begin{align*}\Pfac \tilde{X} v&=(hD_{x}-\Lambda)(hD_x+\Lambda)\tilde{X} v+hE_+\tilde{X} v-R_+ v\\
&= (hD_{x}+\Lambda)(hD_x-\Lambda)\tilde{X} v+hE_-\tilde{X} v-R_- v.
\end{align*}
with
$$
R_\pm = O(h^\infty)_{\psit{W^{1,1}}^{-\infty}}.
$$
Now let $$\begin{aligned}
\Pfac_+ &=(hD_{x}-\Lambda)(hD_x+\Lambda)+hE_+,\\
\Pfac_- &=(hD_{x}+\Lambda)(hD_x-\Lambda)+hE_-.
\end{aligned}$$
Thus
$$
\Pfac_\pm (\tilde{X} v) = \Pfac (\tilde X v) + R_\pm v,
$$
hence Lemma~\ref{l:LinfL2} implies that for all $N \in \NN$,
\begin{align*}
\|\tilde{X}v\|_{L^\infty_xH^1_{h,y}}&\leq C(\|\tilde{X}v\|_{H_h^1}+h^{-1}\|\Pfac \tilde{X}v\|_{L^2}+ h^N \| v\|_{L^2})\\
&\leq C(\|v\|_{H_h^1}+h^{-1}\|\tilde{X}\Pfac v\|_{L^2})\\
&\leq C(\|v\|_{L^2}+h^{-1}\|\Pfac v\|_{L^2}).
\end{align*}

Put $v= e^{\frac{i}{2}\int_0^xa(s,y)ds}u$ so that $\Pfac v= e^{\frac{i}{2}\int_0^xa(s,y)ds}Pu$ and note that, since $a\in L^1((-\e,\e);C^\infty(\mathbb{R}^{d-1}))$,
\begin{align*}
\|Xu\|_{L^\infty_xH_{h,y}^1}\leq C\|\tilde{X}v\|_{L^\infty_xH_{h,y}^1}
&\leq C(\|v\|_{L^2}+Ch^{-1}\|\Pfac v\|_{L^2})\\
&\leq C(\|u\|_{L^2}+Ch^{-1}\|Pu\|_{L^2}).
\end{align*}
In addition, since $a\in L^\infty((-2\e,2\e);C^\infty(\mathbb{R}^{d-1}))$,
\begin{align*}
\|hD_xXu\|_{L^\infty_xH_{h,y}^1}&\leq C(\|hD_x\tilde{X}v\|_{L^\infty_xH_{h,y}^1}+h\|\tilde{X}v\|_{L^\infty_xH_{h,y}^1})\\
&\leq C(\|v\|_{L^2}+Ch^{-1}\|\Pfac v\|_{L^2})\\
&\leq C(\|u\|_{L^2}+Ch^{-1}\|Pu\|_{L^2}).
\end{align*}

\end{proof}

\section{Defect measures for the Schr\"odinger equation}

In this section we continue to assume $V \in W^{1,1}((-2\ep,2\ep); \CI(Y)).$  Our main result here is Lemma~\ref{lemma:hypprop}, which, together with the elliptic estimate in Lemma~\ref{l:ellipticU}, establishes Theorem~\ref{t:hyp}.

We now suppose that $Pu=o(h)_{L^2}$, $\|u\|_{L^2}\leq C$ and study defect measures for $u$.  Recall again that a pure sequence is one along which $\ang{\op a u_h,u_h}$ converges to $\mu(a)$ for a unique defect measure.





\begin{lemma}
\label{l:L1}
Suppose that $\|u\|_{L^2}\leq C$, $Pu=O(h)_{L^2}$, and $u_h$ is a pure sequence with defect measure $\mu$. Then, for $a\in L^1((-\e,\e);C_c^\infty(T^*\mathbb{R}^{d-1}))$, with $\supp a\subset \{r>0\}$, 
\begin{equation}
\label{e:L1}
\limsup_{h\to 0}|\langle\opt(a)u,u\rangle |+|\langle \opt(a)hD_xu,u\rangle| \leq C\|a\|_{L^1_xL^\infty_{y\eta}}.
\end{equation}
In addition, for $a\in L^1((-\e,\e);C_c^\infty(T^*\mathbb{R}^{d-1}))$ with $\supp a\subset \{r>0\}$,
$$
\lim_{h\to 0}\langle \opt(a)u,u\rangle \to \mu(a),\qquad \lim_{h\to 0}\langle \opt(a)hD_{x}u,u\rangle \to \mu(a\xi).
$$
In particular, if $a\in L^1_xL^\infty_{y\eta}$ and $\supp a\subset \{r>0\}$ then $a,\, a\xi\in L^1(\mu)$.
\end{lemma}
\begin{proof}
We start by proving~\eqref{e:L1}. Let $a \in L^1((-\e,\e);C_c^\infty(T^*\mathbb{R}^{d-1})$ with $\supp a \subset \{r>0\}$.  Fix $\chi \in \CI((-\e,\e); C_c^\infty(T^*\RR^{d-1})$ equal to $1$ on $\supp a$ and supported in $\{r>0\}$; thus $\opt(\chi)^* \opt(a)\opt(\chi)=\opt(a)+O(h^\infty)_{L^1_x\Psit^{-\infty}}$. Using Lemma~\ref{l:hyperbolicLinfty} to estimate $\opt(\chi) u$, and Lemma~\ref{l:ellipticU} to estimate $\|u\|_{H_h^2}$, we have

\allowdisplaybreaks
\begin{align*}
&\limsup_{h\to 0}|\langle \opt(a)u,u\rangle|\\
& 
\leq \limsup_{h\to 0}\Big|\int \langle (\opt(\chi)^*\opt(a)\opt(\chi)u)(x), u(x)\rangle_{L^2_y}dx\Big|+\limsup_{h\to 0}O(h^\infty)\|u\|_{L^\infty_xL^2_y}^2\\
& 
=\limsup_{h\to 0}\Big|\int \langle (\opt(a)\opt(\chi)u)(x),(\opt(\chi) u)(x)\rangle_{L^2_y}dx\Big|+\limsup_{h\to 0}O(h^\infty)\|u\|_{H_h^1}^2\\
& \leq \limsup_{h\to 0}\int \|\opt(a)(x)\|_{L_y^2\to L_y^2}\|(\opt(\chi)u)(x)\|_{L_y^2}^2dx\\
&\leq \limsup_{h\to 0}\|\opt(\chi)u\|_{L^\infty_xL^2_y}^2\int \|\opt(a)(x)\|_{L_y^2\to L_y^2}dx\\
&\leq \limsup_{h\to 0}(\|u\|_{L^2}^2+h^{-2}\|Pu\|_{L^2}^2)\int \|\opt(a)(x)\|_{L_y^2\to L_y^2}dx\\
&\leq C\limsup_{h\to 0}\int \|\opt (a)(x)\|_{L_y^2\to L_y^2}dx\\
&\leq C\int \sup_{y,\eta}|\sigma(a)(x,y,\eta)|dx\\
&\leq C\|\sigma(a)\|_{L^1_xL^\infty_{y\eta}}.
\end{align*}

Here we have crucially used Lemma~\ref{l:hyperbolicLinfty} to estimate $\opt(\chi) u$ in terms of $u$ and $Pu$.


By the same line of argument (again using Lemma~\ref{l:hyperbolicLinfty}, as well as the $O(h)$ bounds on commutators)
\begin{align*}
&\limsup_{h\to 0}|\langle \opt(a) hD_x u,u\rangle|\\
&=\limsup_{h\to 0}\Big|\int \langle (\opt(\chi)^*\opt(a)\opt(\chi)hD_x u)(x), u(x)\rangle_{L^2_y}dx\Big| \\
&\quad\quad +\limsup_{h\to 0}O(h^\infty)\|u\|_{L^\infty_xL^2_y}\|hD_xu\|_{L^\infty_xL^2_y}\\
&
\leq \limsup_{h\to 0}\Big|\int \langle (\opt(a)hD_x \opt(\chi) u)(x), \opt(\chi)u(x)\rangle_{L^2_y}dx\Big|\\
&\quad\quad + 
\Big|\int \langle ( \opt(a)[\opt(\chi),hD_x] u)(x), \opt(\chi)u(x)\rangle_{L^2_y}dx\Big|\\ 
&\quad\quad +\limsup_{h\to 0}O(h^\infty)\|u\|^2_{H_h^2}\\
&
=\limsup_{h\to 0}\Big|\int \langle (\opt(a)^* \opt(\chi)u)(x),(hD_x\opt(\chi)) u)(x)\rangle_{L^2_y}dx\Big|\\
&\quad\quad + \limsup_{h\to 0}Ch\|u\|_{L^\infty_xL^2_y}^2\\
&\leq \limsup_{h\to 0}\int \|\opt(a)^*(x)\|_{L_y^2\to L_y^2}\|(h D_x\opt(\chi))u)(x)\|_{L_y^2}\|(\opt(\chi)u)(x)\|_{L^2_y}dx\\
&\leq \limsup_{h\to 0}\|h D_x\opt(\chi)u\|_{L^\infty_xL_y^2}\|\opt(\chi)u\|_{L^\infty_xL^2_y}\int \|\opt(a)^*(x)\|_{L_y^2\to L_y^2}dx\\
&\leq \limsup_{h\to 0}(\|u\|_{L^2}^2+h^{-2}\|Pu\|_{L^2}^2)\int \|\opt(a)^*(x)\|_{L_y^2\to L_y^2}dx\\
&\leq C\|\sigma(a)\|_{L^1_xL^\infty_{y\eta}}.
\end{align*}
This completes the proof of~\eqref{e:L1}.

We now prove the rest of the Lemma. By Lemma~\ref{l:ellipticU}, there is $\chi \in C_c^\infty(T^*M)$ such that $(1-\op(\chi))u=O(h)_{H_h^2}$ and hence
for $a\in C_c^\infty((-\e,\e);C_c^\infty(T^*\mathbb{R}^{d-1}))$, by Lemma~\ref{l:tangentialcomposition},
\begin{equation}
\label{e:defectConv}
\begin{gathered}
\langle \opt(a)u,u\rangle =\langle \opt(a)\op(\chi)u,u\rangle+O(h)\to \mu(a\chi)=\mu(a),\\
\langle \opt(a)hD_xu,u\rangle =\langle \opt(a)hD_x\op(\chi)u,u\rangle+O(h)\to \mu(a\chi\xi)=\mu(a\xi).
\end{gathered}
\end{equation}
Therefore, by~\eqref{e:L1}, 
$$
|\mu(a)|+|\mu(a\xi)|\leq C\|a\|_{L^1_xL^\infty_{y\eta}}.
$$
In particular, by density, if $a\in L^1_xL^\infty_{y\eta}$ with $\supp a\subset \{r>0\}$, then $a,\, a\xi\in L^1(\mu)$. 

Now, let $\psi\in C_c^\infty((-1,1))$ with $\int \psi(x)dx=1$, and define $\psi_\e(x):=\e^{-1}\psi(\e^{-1}x)$, and let $a\in L^1((-\e,\e);C_c^\infty(T^*\mathbb{R}^{d-1}))$ with $\supp a\subset \{r>0\}$. Define $a_\e:=\psi_\e*a$ so that $a_\e \to a$ in $L^1((-\e,\e);C_c^\infty(T^*\mathbb{R}^{d-1})$ and, for $\e>0$ small enough, $\supp a_\e\subset \{r>0\}$. 

Then, by~\eqref{e:L1}
$$
\lim_{\e \to 0}\limsup_{h\to 0}|\langle (\opt(a_\e)-\opt(a))u,u\rangle|+|\langle (\opt(a_\e)-\opt(a))hD_xu,u\rangle|=0
$$
and by~\eqref{e:defectConv}
$$
\lim_{h\to 0}\langle \opt(a_\e)u,u\rangle =\mu(a_\e),\qquad \lim_{h\to 0}\langle \opt(a_\e)hD_xu,u\rangle =\mu(a_\e\xi).
$$
Therefore,  since 
$$\{a\in L^1_xL^\infty_{y\eta}\,:\, \supp a\subset \{r>0\}\}\subset L^1(\mu),$$
we have
$$
\lim_{h\to 0}\langle \opt(a)u,u\rangle =\lim_{\e \to 0}\mu(a_\e)=\mu(a),\qquad \lim_{h\to 0}\langle \opt(a)hD_xu,u\rangle =\lim_{\e \to 0}\mu(a_\e\xi)=\mu(a\xi).
$$

\end{proof}

We now compute $\mu(H_pa)$ for certain special test operators $a$.
\begin{lemma}
\label{l:specialTest}
Suppose that $\|u\|_{L^2}\leq C$, $Pu=o(h)_{L^2}$, and $u_h$ is a pure sequence with defect measure $\mu$. Then
for $a_j\in C_c^\infty((-\e,\e);C_c^\infty(T^*\mathbb{R}^{d-1}))$ with $\supp a_j\subset \{r>0\}$,
$$
\mu(H_p(a_0+a_1\xi))=0.
$$
\end{lemma}
\begin{proof}
First, observe that since, by Lemma~\ref{l:ellipticU}, $\|u\|_{H_h^1}\leq C$, 
\begin{align*}
|\langle [P,\opt(a_0)&+\opt(a_1)hD_x]u,u\rangle|\\
&\leq  |\langle (\opt(a_0)+\opt(a_1)hD_{x_1})u,Pu\rangle| +|\langle (\opt(a_0)+\opt(a_1)hD_{x_1})Pu,u\rangle|\\
&\leq C\|u\|_{H_h^1}\|Pu\|_{L^2}=o(h).
\end{align*}

Now, for $a\in C^\infty((-\e,\e);C_c^{\infty}(T^*\mathbb{R}^{d-1}))$
\begin{gather*}
ih^{-1}[P,\opt(a)]=2\opt(\partial_xa)hD_{x}-\opt(\{r,a\})+h\opt(e_1) +O(h^\infty)_{\psit{W^{1,1}}^{-\infty}}\\
ih^{-1}[P,hD_{x}]=\opt(\partial_x r)+h\opt(e_2)hD_{x}+O(h^\infty)_{\psit{W^{1,1}}^{-\infty}}
\end{gather*}
for some $e_i\in W^{1,1}((-\e,\e);C_c^\infty(T^*\mathbb{R}^{d-1}))$ with $\supp e_i\subset \supp a$. 
Consequently,
\begin{align*}
&ih^{-1}\langle [P,\opt(a_0)+\opt(a_1)hD_x]u,u\rangle\\
&= ih^{-1}\langle ([P,\opt(a_0)]+[P,\opt(a_1)]hD_x+\opt(a_1)[P,hD_{x}])u,u\rangle\\
&=\big \langle\big( 2\opt(\partial_x a_0)hD_x-\opt(\{r,a_0\})\big)u,u\big\rangle\\
&\qquad+\big\langle\big((2\opt(\partial_x a_1)hD_x-\opt(\{r,a_1\}))hD_{x}+\opt(a_1)\opt(\partial_xr)\big)u,u\big\rangle\\
&\qquad+h\langle(\opt(b_0)+\opt(b_1)hD_x)u,u\rangle\\
&=\big \langle\big( 2\opt(\partial_x a_0)hD_x-\opt(\{r,a_0\})\big)u,u\big\rangle\\
&\qquad+\big\langle\big((-\opt(\{r,a_1\}))hD_{x}+\opt(a_1)\opt(\partial_xr)\big)u,u\big\rangle+\big\langle2\opt(\partial_x a_1)(P+\opt(r))u,u\big\rangle\\
&\qquad+h\langle(\opt(b_0)+\opt(b_2)hD_x)u,u\rangle
\end{align*}
with $b_i\in L^1((-\e,\e);C_c^\infty(T^*\mathbb{R}^{d-1}))$ and $\supp b_i\subset  \{r>0\}$.

Using Lemma~\ref{l:L1}, together with the fact that $\|u\|_{H_h^1}\leq C$, we obtain
\begin{align*}
0&=\lim_{h\to 0}ih^{-1}\langle [P,\opt(a_0)+\opt(a_1)hD_x]u,u\rangle\\
&=\mu(2\partial_xa_0\xi-\{r,a_0\}-\{r,a_1\}\xi+a_1\partial_xr+2\partial_xa_1(p+r))\\
&=\mu(H_p(a_0+a_1\xi)),
\end{align*}
and the lemma is proved.
\end{proof}

Finally, we extend the previous lemma to any test function $a\in C_c^\infty(T^*M)$.

\begin{lemma}\label{lemma:hypprop}
Suppose that $\|u\|_{L^2}\leq C$, $\|Pu\|_{L^2}=o(h)$, and $u$ has defect measure $\mu$. Then for all $a\in C_c^\infty (T^*M)$ supported close enough to $\{x=0\}$ with $\supp a\subset \{r>0\}$,
$$
\mu(H_pa)=0. 
$$
\end{lemma}
\begin{proof}
Define
$$
a_e:=\tfrac{1}{2}(a(x,\xi,y,\eta)+a(x,-\xi,y,\eta)),\qquad a_o:=\tfrac{1}{2\xi}(a(x,\xi,y,\eta)-a(x,-\xi,y,\eta)).
$$
Then $a_e,\, a_o\in C_c^\infty(T^*M)$ and both are even in $\xi$. Moreover, 
$$
a=a_e+a_o \xi.
$$
Since $a_e,\,a_o$ are even in $\xi$, there are $\tilde{a}_e,\, \tilde{a}_o\in C_c^\infty(T^*M)$ such that
$$
a_{e/o}(x,\xi,y,\eta)=\tilde{a}_{e/o}(x,\xi^2,y,\eta). 
$$
Finally, put
$$
b_{e/o}(x,y,\eta):=\tilde{a}_{e/o}(x,r(x,y,\eta),y,\eta).
$$
Then, $b_{e/o}\in W^{1,1}((-\e,\e);C_c^\infty(\mathbb{R}^{2n-1}_{\xi,y,\eta}))$ and $b_{e/o}=a_{e/o}$ on $\{p=0\}$.

Now note that by Taylor's theorem (initially treating $r$ as an independent variable),
$$\begin{aligned}
\tilde{a}_{e/o}(x,\xi^2,y,\eta) &= \tilde{a}_{e/o}(x,r(x,y,\eta),y,\eta)+(\xi^2-r(x,y,\eta)) g_{e/o}(x,\xi^2,y,\eta,r(x,y,\eta))\\ &=
b_{e/o}(x,y,\eta)+(\xi^2-r(x,y,\eta)) g_{e/o}(x,\xi^2,y,\eta,r(x,y,\eta)
)
\end{aligned}
$$
with $g_{e/o}$ smooth in all its arguments.  Hence
$$
a=\tilde{a}_e(x,\xi^2,y,\eta)+\xi \tilde{a}_o (x,\xi^2,y,\eta) = b_e(x,y,\eta)+\xi b_o(x,y,\eta)+p g(x,\xi^2,y,\eta,r(x,y,\eta)),
$$



Hence, we have
$$
H_pa|_{p=0}=[H_p(b_e+b_o\xi)]|_{p=0}.
$$

Therefore, since $\supp \mu \subset \{p=0\}$ by Lemma~\ref{l:ellipticU}, this implies
$$
\mu(H_pa)=\mu(H_p(b_e+b_o\xi)),
$$
and hence the lemma follows from Lemma~\ref{l:specialTest}.
\end{proof}

\section{Propagation for $\mathcal{C}^1$ potentials}

We now focus on the simpler case when $V\in \mathcal{C}^1$. In this case, it is not necessary to use special factorization structure, and one can apply directly the standard arguments for invariance of defect measures. Although the results in this section can be obtained from~\cite{BuDeLeRo:22}, we give a simple self contained proof in the semiclassical setting.  This, in conjunction with the elliptic estimate of Lemma~\ref{l:ellipticU}, will establish the propagation estimate of Theorem~\ref{theorem:C1}.

\begin{lemma}
\label{l:commute}
Suppose that $a\in C_c^\infty(T^*M)$. Then there is $C>0$ such that for all $V\in W^{1,\infty}$,
$$
\|[\op(a),V]\|_{L^2\to L^2}\leq Ch\|V\|_{W^{1,\infty}}.
$$
\end{lemma}
\begin{proof}
First, observe that we may work locally since for $\chi,\phi\in C^\infty(M)$, with $\supp \chi\cap \supp \phi=\emptyset$, for any $N$, there is $C_N>0$ such that 
$$
\|\chi \op(a)\psi\|_{L^2\to L^2}\leq C_Nh^N,
$$
and hence
$$
\|[\chi\op(a)\psi,V]\|_{L^2\to L^2}\leq Ch\|V\|_{L^\infty}\leq Ch\|V\|_{W^{1,\infty}}.
$$
Therefore,  after decomposing using a partition of unity we may replace $a$ by $\tilde{\chi}\op(a)\chi$ for some $\chi,\tilde{\chi}\in C_c^\infty(M)$ with $\supp \chi\cap \supp (1-\tilde{\chi})=\emptyset$, and $\tilde{\chi}$ supported in a coordinate patch. 

In local coordinates, the kernel of $[\tilde{\chi}\op(a)\chi,V]$ is given in local coordinates by
$$
K(x,y):=\frac{1}{(2\pi h)^d}\int e^{\frac{i}{h}\langle x-y,\xi\rangle}\tilde{\chi}(x)\tilde{a}(x,\xi)\chi(y)(V(y)-V(x))d\xi,
$$
for some $\tilde{a}\in C_c^\infty(T^*\mathbb{R}^d)$. 
Then, integrating by parts once in $\xi$, we obtain
$$
K(x,y):=\frac{1}{i(2\pi h)^d}\int e^{\frac{i}{h}\langle x-y,\xi\rangle}\frac{\langle y-x,h\partial_\xi \tilde{a}(x,\xi)\rangle}{|x-y|^2}\tilde{\chi}(x)\chi(y)(V(y)-V(x))d\xi.
$$
Then, integrating by parts with $L:=\frac{h+\langle x-y,D_\xi\rangle}{h+h^{-1}|x-y|^2}$, we obtain
$$
K(x,y):=\frac{1}{i(2\pi h)^d}\int e^{\frac{i}{h}\langle x-y,\xi\rangle}\Big(\frac{h+\langle y-x,D_\xi\rangle}{h+h^{-1}|x-y|^2}\Big) ^N\frac{\langle y-x,h\partial_\xi a(x,\xi)\rangle}{|x-y|^2}\tilde{\chi}(x)\chi(y)(V(y)-V(x))d\xi.
$$
In particular, 
$$
|K(x,y)|\leq Ch^{1-d}\langle h^{-1}|x-y|\rangle^{-N}\frac{|V(y)-V(x)|}{|x-y|}\leq Ch^{1-d}\langle h^{-1}|x-y|\rangle^{-N}\|V\|_{W^{1,\infty}}
$$
So that
$$
\sup_x\int|K(x,y)|+\sup_y\int|K(x,y)|\leq Ch\|V\|_{W^{1,\infty}}.
$$
The Schur test for $L^2$ boundedness then implies the lemma.
\end{proof}

\begin{lemma}\label{lemma:C1prop}
Let $V\in \mathcal{C}^1$. Then if $u\in L^2(M)$ solves 
$$
\|(-h^2\Delta_g+V)u\|_{L^2}=o(h)_{L^2},\qquad \|u\|_{L^2}\leq C<\infty,
$$
and has defect measure $\mu$. Then for all $a\in C_c^\infty(T^*M)$, 
$$
\mu(H_pa)=0.
$$
\end{lemma}
\begin{proof}
 Let $\{\chi_i\}_{i=1}^N\subset C_c^\infty(M)$ be a partition of unity on $M$ with $\supp \chi_i\subset U_i$ and $(\psi_i:U_i\to \mathbb{R}^d,U_i)$ a coordinate system on $M$. Then, let $\phi \in C_c^\infty(\mathbb{R}^d)$ with $\int \phi=1$, define $\phi_\e(x):=\e^{-d}\phi(\e^{-1}x)$, and put 
$$
V_\e:= \sum_i\big[(\chi_i V)\circ \psi_i^{-1}*\phi_\e\big]\circ \psi_i
$$
Then,
$$
\| V_{\e}\|_{\mathcal{C}^1}\leq C,\qquad \lim_{\e\to 0}\|V_\e-V\|_{W^{1,\infty}}=0.
$$

Let $p_{\e}=|\xi|_g^2+V_{\e}$. Then for $a\in C_c^\infty(T^*M)$ real valued, we have
\begin{align*}
0&=\lim_{h\to 0}2h^{-1}\Im \langle Pu,\op(a)u\rangle \\
&=\lim_{h\to 0}-ih^{-1}(\langle \op(a)^*Pu,u\rangle-\langle P\op(a)u,u\rangle\\
&=\lim_{h\to 0}-ih^{-1}(\langle \op(a)Pu,u\rangle-\langle P\op(a)u,u\rangle\\
&=\lim_{h\to 0}ih^{-1}\langle [P,\op(a)]u,u\rangle\\
&=\lim_{\e \to 0}\lim_{h\to 0}ih^{-1}\langle [-h^2\Delta_g+V_{\e},\op(a)]u,u\rangle+ih^{-1}\langle [V-V_{\e},\op(a)]u,u\rangle.
\end{align*}
Notice that
$$
\lim_{h\to 0}ih^{-1}\langle [-h^2\Delta_g+V_{\e},\op(a)]u,u\rangle=\lim_{h\to 0}\langle \op(H_{p_{\e}})u,u\rangle=\mu(H_{p_\e}).
$$
For the second term, observe that by Lemma~\ref{l:commute}
$$
|ih^{-1}\langle [V-V_{\e},\op(a)]u,u\rangle|\leq C\|V-V_\e\|_{W^{1,\infty}},
$$
and hence 
$$
\lim_{\e\to 0}\lim_{h\to 0}ih^{-1}\langle [V-V_{\e},\op(a)]u,u\rangle=0.
$$
All together, we have shown that 
$$
0=\lim_{\e\to 0}\mu(H_{p_\e}a)
$$

On the other hand, by dominated convergence, 
$$ 
\lim_{\e\to 0}\mu(H_{p_\e}a)=\mu(H_pa),
$$
which completes the proof.
\end{proof}

      \bibliographystyle{abbrv} 
\bibliography{references.bib}

\end{document}